\documentclass{article}
\usepackage{graphicx} 
\usepackage{amsthm}
\usepackage{amsfonts, amsmath, amssymb, amsthm}
\usepackage{cleveref}
\usepackage[utf8x]{inputenc}
\usepackage{url}

\def\PP{\mathbb{P}}

\def\Tc{\mathcal{T}}
\def\Sc{\mathcal{S}}

\def\J{\mathcal{J}}
\def\I{\mathcal{I}}

\def\Rc{\mathcal{R}}

\newcommand{\uh}{\upharpoonright}

\newcommand{\qvdash}{\operatorname{{?}{\vdash}}}
\newcommand{\nqvdash}{\operatorname{{?}{\nvdash}}}

\newcommand{\BSig}{\mathsf{B}\Sigma^0}

\newcommand{\RCA}[0]{\mathsf{RCA}}

\newcommand{\ACA}[0]{\mathsf{ACA}}

\newcommand{\RT}[0]{\mathsf{RT}}

\newcommand{\HT}[0]{\mathsf{HT}}
\newcommand{\FUT}[0]{\mathsf{FUT}}

\newcommand{\NN}[0]{\mathbb{N}}

\newcommand{\SG}{\operatorname{SG}} 
\newcommand{\VSG}{\operatorname{VSG}} 
\newcommand{\card}{\operatorname{card}}
\newcommand{\dom}{\operatorname{dom}}
\newcommand{\ran}{\operatorname{ran}}

\newcommand{\Pfin}{\mathcal{P}_{\mathtt{fin}}}
\newcommand{\us}{\operatorname{us}}
\newcommand{\FS}[0]{\mathsf{FS}} 
\newcommand{\FU}[0]{\mathsf{FU}} 

\def\qt#1{``#1''}%
\def\h#1{\hat{#1}}

\title{Hindman's theorem does not code $\emptyset^{(\omega)}$\\ in one application}
\date{\today}

\newtheorem{theorem}{Theorem}
\numberwithin{theorem}{section}
\newtheorem{maintheorem}[theorem]{Main Theorem}
\newtheorem{lemma}[theorem]{Lemma}

\newtheorem{proposition}[theorem]{Proposition}
\newtheorem{remark}[theorem]{Remark}
\newtheorem{definition}[theorem]{Definition}

\newtheorem{claim}{Claim}

\AtBeginEnvironment{theorem}{\setcounter{claim}{0}}
\AtBeginEnvironment{maintheorem}{\setcounter{claim}{0}}
\AtBeginEnvironment{lemma}{\setcounter{claim}{0}}
\AtBeginEnvironment{proposition}{\setcounter{claim}{0}}
\AtBeginEnvironment{corollary}{\setcounter{claim}{0}}
\AtBeginEnvironment{question}{\setcounter{claim}{0}}
\AtBeginEnvironment{example}{\setcounter{claim}{0}}

\makeatletter
\newtheorem*{rep@theorem}{\rep@title}
\newcommand{\newreptheorem}[2]{%
  \newenvironment{rep#1}[2][]{%
    \def\optarg{##1}%
    \ifx\optarg\@empty
      \def\rep@title{#2 \ref{##2}}%
    \else
      \def\rep@title{#2 \ref{##2} (##1)}%
    \fi
    \begin{rep@theorem}}%
    {\end{rep@theorem}}}
\makeatother

\newreptheorem{theorem}{Theorem}
\newreptheorem{maintheorem}{Main Theorem}

\usepackage{authblk}
\DeclareSymbolFont{bbold}{U}{bbold}{m}{n}
\DeclareMathSymbol{\bbomega}{\mathord}{bbold}{"7F}

\usepackage{xcolor}

\author{Lu Liu and Ludovic Patey}

\begin{document}

\maketitle

\begin{abstract}
We prove that for every non-arithmetic set~$C$ and every arithmetic finite coloring of~$\NN$, there is an infinite set $H \subseteq \NN$ whose non-empty finite sums of distinct elements is monochromatic, and $C$ is not $H$-computable. We also study restrictions of Hindman's theorem to simple colorings.
\end{abstract}

\section{Introduction}

In 1974, Hindman~\cite{hindmanFiniteSumsSequences1974} proved the following theorem, which was conjectured by Graham and Rothschild:

\begin{theorem}[Hindman]
For every~$\ell \geq 1$ and every coloring $f : \NN \to \ell$, there is an infinite set $H \subseteq \NN$ such that $\FS(H)$ is $f$-monochromatic, where $\FS(H)$ is the set of all non-empty finite sums of distinct elements over~$H$.
\end{theorem}

Hindman's theorem admits several proofs, including a short one by Baumgartner~\cite{baumgartnerShortProofHindmans1974}, an ultrafilter proof by Galvin and Glazer (see Comfort~\cite{comfortUltrafiltersOldNew1977} or Hindman and Strauss~\cite{hindmanAlgebraStoneCechCompactification2012}), one in topological dynamics by Furstenberg and Weiss~\cite{furstenberg1978topological}, and, more recently, a simple combinatorial proof by Towsner~\cite{towsnerSimpleProofDifficult2012}.

The theorem received a great amount of attention from a meta-mathematical viewpoint. Hindman's original proof, its topological dynamics version and the simple proof by Baumgartner were analyzed by Blass, Hirst and Simpson~\cite{blassLogicalAnalysisTheorems1987} in reverse mathematics. They showed that the two former proofs hold over~$\ACA_0^+$, while the latter holds over the much stronger system $\Pi^1_2\mbox{-}\mathsf{TI}_0$. On the other hand, they proved that Hindman's theorem implies $\ACA_0$ over~$\RCA_0$. Towsner~\cite{towsner2011hindman} showed that the Galvin and Glazer's ultrafilter proof holds in $\Sigma^1_1\mbox{-}\mathsf{TI}_0$. Last, Towsner's simple combinatorial proof~\cite{towsnerSimpleProofDifficult2012} holds over~$\ACA_0^+$. 

In order to better grasp the complexity of Hindman's theorem, restricted versions were also studied in a reverse-mathematical setting~\cite{carlucciWeakStrongRestrictions2016,carlucciWeakVariantHindman2017,carlucci2020new,dzhafarovEffectivenessHindmanTheorem2017}. In particular, Hindman, Leader and Strauss~\cite{hindman2003open} asked for a proof of Hindman's theorem restricted to sums of length at most~2 ($\HT^{\leq 2}$) which does not produce a full solution to Hindman's theorem. Bompard, Liu and Patey~\cite{bompardReverseMathematicsCarlson2022} gave a partial answer by proving that the iterated version of $\HT^{\leq 2}$ implies Hindman's theorem over~$\RCA_0$.

From a purely computability-theoretic viewpoint, the best known upper bound to date was proven by Blass, Hirst and Simpson~\cite{blassLogicalAnalysisTheorems1987} in 1987 :

\begin{theorem}[Blass, Hirst and Simpson]
For every~$\ell \geq 1$ and every computable coloring $f : \NN \to \ell$, there is a $\emptyset^{(\omega)}$-computable infinite set $H \subseteq \NN$ such that $\FS(H)$ is $f$-monochromatic.
\end{theorem}

Here, $X^{(\omega)}$ is the $\omega$-jump of~$X$, defined inductively from the Turing jump operator $X \mapsto X' = \{ e : \Phi^X_e(e)\downarrow \}$ as follows: $X^{(0)} = X$, $X^{(n+1)} = (X^{(n)})'$, and $X^{(\omega)} = \bigoplus_n X^{(n)}$. The unique positive lower bound was proven in the same paper:

\begin{theorem}[Blass, Hirst and Simpson]
There is a computable coloring $f : \NN \to 2$ such that for every infinite set~$H \subseteq \NN$ 
such that $\FS(H)$ is $f$-monochromatic, $H$ computes the halting set.
\end{theorem}

In 2026, Liao~\cite{liao2026recursive} improved a negative $\Delta^0_2$ lower bound by Blass, Hirst and Simpson, yielding the following $\Sigma^0_2$ negative lower bound:

\begin{theorem}[Liao]
There is a computable coloring $f : \NN \to 2$ with no $\Pi^0_3$ infinite set~$H \subseteq \NN$ such that $\FS(H)$ is $f$-monochromatic.
\end{theorem}

In this article, we show the non-optimality of the computability-theoretic upper bound by proving that Hindman's theorem admits cone avoidance for non-arithmetic cones. More precisely, we prove the following theorem:

\begin{maintheorem}\label[maintheorem]{main:cone-avoidance}
Let $C$ be a set of non-arithmetic degree.
For every $\ell \geq 1$ and every coloring $f : \NN \to \ell$ of arithmetic degree,
there is an infinite set $H \subseteq \NN$ such that $\FS(H)$ is $f$-monochromatic, and $C \not \leq_T H$.
\end{maintheorem}

In particular, letting $C$ be $\emptyset^{(\omega)}$, this shows that every arithmetic instance of Hindman's theorem admits an infinite solution which does not compute $\emptyset^{(\omega)}$.
Note that this does not rule out the possibility that Hindman's theorem is equivalent to $\ACA_0^+$ using multiple applications. Indeed, \Cref{main:cone-avoidance} cannot be iterated to build an $\omega$-model of Hindman's theorem not containing $\emptyset^{(\omega)}$, as the conclusion only holds for Turing reducibility instead of arithmetic reducibility.

\Cref{main:cone-avoidance} is obtained by adapting the combinatorics of Towsner~\cite{towsnerSimpleProofDifficult2012} to obtain a notion of forcing with good first-jump control to produce solutions to Hindman's theorem. The approach was first used by the authors in~\cite{liu2026reverse} to transform a Towsner-like proof of a variable word theorem by Carlson and Simspon~\cite{carlson1984dual} into a notion of forcing with good first-jump control.

Motivated by the combinatorial features of the coloring witnessing the lower bound of Blass, Hirst and Simpson~\cite{blassLogicalAnalysisTheorems1987}, we also study a class of colorings, called \emph{simple colorings}, and prove that Hindman's theorem restricted to this class is equivalent to~$\ACA_0$ over~$\RCA_0$.

\subsection{Finite Union Theorem}

Finite sums of distinct elements is not a closure property, in the sense that if $Y \subseteq \FS(X)$, then $\FS(Y)$ is not in general a subset of $\FS(X)$. For these reasons, Hindman's theorem is often more conveniently formulated in terms of finite unions of finite sets. The corresponding statement, known as the Finite Union Theorem, is computably equivalent to Hindman's theorem and this equivalence holds over~$\RCA_0$. Therefore, all the computability-theoretic and reverse-mathematical bounds can be stated indistinctly under both formalisms.

We write $\Pfin(\NN)$ for the collection of all non-empty finite subsets of~$\NN$, and use small letters $a, b, c$ to denote its members. A  \emph{disjoint sequence} is a non-empty (finite or infinite) sequence of non-empty finite sets $X = \{ a_0, a_1, \dots \}$ such that for every~$i < j$, $a_i \cap a_j = \emptyset$. It is a \emph{block sequence} if furthermore $\max a_i < \min a_j$.
Given a disjoint sequence $X = \{a_0, a_1, \dots \}$, write $\FU(X)$ for the collection of all non-empty finite unions of $X$.

\begin{theorem}[Finite Union Theorem]\label[theorem]{thm:fut}
For every~$\ell \geq 1$, and every coloring $f : \Pfin(\NN) \to \ell$, there is an infinite block sequence $H \subseteq \Pfin(\NN)$ such that $\FU(H)$ is $f$-homogeneous.
\end{theorem}

We shall therefore prove our main theorem in the following form:

\begin{repmaintheorem}[reformulated]{main:cone-avoidance}
Let $C$ be a set of non-arithmetic degree.
For every $\ell \geq 1$ and every coloring $f : \Pfin(\NN) \to \ell$ of arithmetic degree,
there is an infinite block sequence $H \subseteq \Pfin(\NN)$ such that $\FU(H)$ is $f$-monochromatic, and $C \not \leq_T H$.
\end{repmaintheorem}

\subsection{Notation}

Given two sets~$a, b \subseteq \Pfin(\NN)$, we write $a < b$ if $\max a < \min b$.
Given two (finite or infinite) sets $F, X \subseteq \Pfin(\NN)$, we write $F < X$ if for every~$a \in F$ and $b \in X$, $a < b$.

Let $F, X$ be two block sequences. We write $F \preceq X$ if $F$ is an initial segment of~$X$, that is, $X = F \cup Y$ for some block sequence $Y > F$. 
Given an infinite block sequence $X = \{ a_0 < a_1 < \cdots \} \subseteq \Pfin(\NN)$ and a finite set $b \in \Pfin(\NN)$, we write $X[b]$ for the set $\bigcup_{s \in b} a_s$. Accordingly, given a finite or infinite block sequence $F \subseteq \Pfin(\NN)$, we write $X[F]$ for the block sequence $\{ X[b] : b \in F \}$. Note that $\FU(X[F]) = X[\FU(F)]$.

\section{Avoiding non-arithmetic cones}

The goal of this section is to prove \Cref{main:cone-avoidance}, that is, cone avoidance of the Finite Union Theorem for non-arithmetic cones.
In what follows, fix an arithmetic coloring $f : \Pfin(\NN) \to \ell$ for some~$\ell \geq 1$.
We first recall in \Cref{sec:core-combinatorics} the core combinatorics from Towsner~\cite{towsnerSimpleProofDifficult2012}, and study their computability-theoretic bounds in~\Cref{sec:arithmetic-bounds}, before using them to design a notion of forcing in \Cref{sec:notion-of-forcing}.

\subsection{Core combinatorics}\label[section]{sec:core-combinatorics}

Towsner introduced the two core notions of half-match and full-match, the latter notion providing a sufficiently large set of elements, one of which being extensible into an infinite solution.

\begin{definition}[Towsner~\cite{towsnerSimpleProofDifficult2012}]
Let $F \subseteq \Pfin(\NN)$ be a finite set and $X$ be an infinite block sequence such that $F < X$.
\begin{itemize}
    \item[(1)] $F$ \emph{half-matches} $X$ if for every~$b \in \FU(X)$, there is some~$a \in F$ such that $f(a \cup b) = f(b)$.
    \item[(2)] $F$ \emph{full-matches} $X$ if for every~$b \in \FU(X)$, there is some~$a \in F$ such that $f(a) = f(a \cup b) = f(b)$.
\end{itemize}
\end{definition}

If $F$ half-matches $X$ for~$f$, then the \emph{witness function} associated to $(\ell, f, F, X)$ is the function $\hat f : \FU(X) \to \ell \times F$ defined by:
\begin{align*}
\hat f : b & \mapsto (f(b), a)\\
& \text{such that } f(a \cup b) = f(b).
\end{align*}
If there is more than one corresponding set~$a$, we choose the least one, for some fixed order. Before going any further, we justify the notion of full-match by proving that, under the assumption that the Finite Union Theorem classically holds, at least one element of~$F$ is extensible into an infinite solution. The argument using $\FUT$, it is only used to support the intuition of the definition:

\begin{lemma}
Let $F \subseteq \Pfin(\NN)$ be a finite set and $X$ an infinite block sequence such that $F$ full-matches~$X$.
There is an infinite block sequence~$H \subseteq F \cup \FU(X)$ such that $H \cap F \neq \emptyset$ and $\FU(H)$ is $f$-monochromatic. 
\end{lemma}
\begin{proof}
Let $\bar f : \Pfin(\NN) \to \ell \times F$ be defined by:
\begin{align}\label{eq:full-match-eq1}
\bar f : b & \mapsto (f(a), a) \notag \\
& \text{such that } f(a) = f(a \cup X[b]) = f(X[b]).
\end{align}
Such an $a \in F$ exists since $F$ full-matches~$X$.
By $\FUT$, there is an infinite block sequence $Y \subseteq \Pfin(\NN)$ such that $\FU(Y)$ is $\bar f$-monochromatic for some color $(\ell^*, a) \in \ell \times F$. In other words, by (\ref{eq:full-match-eq1}), for every~$b \in \FU(Y)$,
\begin{align}\label{eq:full-match-eq2}
f(a) = f(a \cup X[b]) = f(X[b]) = \ell^*
\end{align}
Let $H = \{a\} \cup X[Y]$. For any element~$c \in \FU(H)$, either $c = a$, or $c = X[b]$ for some $b \in \FU(Y)$, or $c = a \cup X[b]$ for some~$b \in \FU(Y)$. In either case, by (\ref{eq:full-match-eq2}), $f(c) = \ell^*$. This completes our proof of the lemma.
\end{proof}

In order to construct a full-match, Towsner defined a nested sequence of half-matches, inducing a tree that we shall call a \emph{Towsner tree}, and that will play an essential role in the first-jump control of our notion of forcing. This corresponds to the Graham-Rothschild tree in~\cite{liu2026reverse}.

\begin{definition}
A \emph{Towsner sequence} for~$f$ is a sequence 
$$\Sc = \{ (\ell_0, f_0, F_0, X_0), (\ell_1, f_1, F_1, X_1), \cdots \}$$
such that 
\begin{itemize}
    \item[(1)] $F_n \subseteq \Pfin(\NN)$ is a finite set; $X_n$ is an infinite block sequence; $F_n < X_n$ and $F_{n+1} \cup X_{n+1} \subseteq \FU(X_n)$.
    \item[(2)] $\ell_0 = \ell$ ; $\ell_{n+1} = \ell_n \times F_n$;
    \item[(3)] $F_n$ half-matches $X_n$ for $f_n$;
    \item[(4)] $f_0 = f$ ; $f_{n+1} : \FU(X_n) \to \ell_{n+1}$ is the witness function associated to $(\ell_n, f_n, F_n, X_n)$.
\end{itemize}
\end{definition}

\begin{definition}
Given a Towsner sequence
$$\Sc = \{(\ell_0, f_0, F_0, X_0), (\ell_1, f_1, F_1, X_1), \cdots \}$$
its \emph{Towsner tree} $\Tc_{\Sc}$ is the tree
$$
\bigcup_{n \in \NN} \prod_{i \leq n} F_i
$$
In other words, it is the tree of finite block sequences~$E = \{ a_0 < \dots < a_{n-1} \}$ such that for every $i < n$, $a_i \in F_i$.
\end{definition}

The core property of the Towsner tree is the following:

\begin{proposition}\label[proposition]{prop:towsner-tree}
Let $\Tc$ be the Towsner tree associated to a Towsner sequence
$$\Sc = \{ (\ell_0, f_0, F_0, X_0), (\ell_1, f_1, F_1, X_1), \cdots \}$$
For every~$n \in \NN$ and every $b \in \FU(X_n)$, there is a block sequence $E$ at level~$n$ in $\Tc$ (in other words $E \in \prod_{i \leq n} F_i$) such that for every~$a \in \FU(E)$, $f(a \cup b) = f(b)$.
\end{proposition}
\begin{proof}
We will prove the following by induction on $n$:
\begin{align}\label{eq:towsner-tree-conclusion}
&\text{let $b\in \FU(X_n)$ and suppose
$ 
f_{n+1}(b) = (\ell^*, a_0, a_1, \cdots, a_n)$,}
\\ \nonumber
&\text{then for every~$c \in \FU(a_0,\cdots, a_n)\cup\{\emptyset\}$,
we have  $f(c \cup b)   = \ell^*$.
}
\end{align}

For $n=0$, let $b\in \FU(X_0)$. Say,
$ 
f_1(b) = (\ell^*, a_0)
$ 
where $\ell^*<\ell$ and $a_0\in F_0$.
Since $f_1$ is the witness function associated to $(\ell_0, f_0, F_0, X_0)$
(and $f_0 = f$),
we have
$$
  f(a_0 \cup b) = f(b) = \ell^*.
$$
This verifies (\ref{eq:towsner-tree-conclusion}) for $n = 0$.

Suppose the conclusion holds for $n\leq n-1$ and now $n = n$.
Let $b\in \FU(X_n)$. Say,
$$
f_{n+1}(b) = (\ell^*, a_0, a_1, \cdots, a_n).
$$
Since $f_{n+1}$ is the witness function associated to $(\ell_n, f_n, F_n, X_n)$,
we have
$$
f_n(a_n \cup b) = f_n(b) = (\ell^*,a_0,\cdots, a_{n-1}).
$$
Combine with induction hypothesis, we have: for every 
$c \in \FU(a_0,\cdots, a_{n-1})\cup\{\emptyset\}$,
$$
f(c \cup a_n \cup b)   =
f(c \cup b) =   \ell^*.$$
This verifies (\ref{eq:towsner-tree-conclusion}) for $n$.

\end{proof}

\subsection{Arithmetic bounds}\label[section]{sec:arithmetic-bounds}

The original article of Towsner~\cite{towsnerSimpleProofDifficult2012} is motivated by reverse mathematics, and therefore formulated in terms of subsystems of second-order arithmetics. However, thanks to the correspondence between computability and definability, his analysis can be straightforwardly recasted in the computability-theoretic setting, with the following bounds:

\begin{proposition}[Towsner~\cite{towsnerSimpleProofDifficult2012}]\label[proposition]{prop:half-matches-comp}
For every coloring $f : \Pfin(\NN) \to \ell$ with $\ell \geq 1$,
there is a finite set~$F$ and an $f$-computable block sequence $X$
such that $F$ half-matches $X$ for~$f$.
\end{proposition}

\begin{proposition}[Towsner~\cite{towsnerSimpleProofDifficult2012}]\label[proposition]{prop:full-matches-arith}
For every coloring $f : \Pfin(\NN) \to \ell$ with $\ell \geq 1$,
there is a finite set~$F$ and an $f$-arithmetic block sequence $X$
such that $F$ full-matches $X$ for~$f$.
\end{proposition}

Using the existence of an $f$-computable half-match as a blackbox, one can easily construct an $f$-arithmetic Towsner sequence by searching $f''$-computably for Turing indices witnessing the existence of an $f$-computable half-match. Indeed, the statement \qt{$e$ codes an $f$-computable half-match for~$f$} is $\Pi^0_2(f)$.

\begin{proposition}\label[proposition]{prop:towsner-seq-arith}
For every coloring $f : \Pfin(\NN) \to \ell$ with $\ell \geq 1$,
there is an $f''$-computable Towsner sequence.
\end{proposition}
\begin{proof}
A \emph{code} for an $f$-computable set~$X$ is a Turing index~$e$ such that $\Phi_e^f = X$.

Let $X_{-1} = \Pfin(\NN)$, $\ell_0 = \ell$, and $f_0 = f$.
Having defined an $f$-computable block sequence $X_{n-1}$, some $\ell_n \in \NN$ and an $f$-computable coloring $f_n : \FU(X_{n-1})\allowbreak \to \ell_n$, let $F_n \subseteq \FU(X_{n-1})$ be a finite set and $X_n \subseteq \FU(X_{n-1})$ be an infinite $f$-computable block sequence such that $F_n$ half-matches $X_n$ for~$f_n$. Such a pair exists by \Cref{prop:half-matches-comp}. Moreover, $F_n$ and a code for $X_n$ can be $f''$-computably found uniformly in a code for~$X_{n-1}$, $\ell_n$ and a code for $f_n$. Last, let $\ell_{n+1} = \ell_n \times F_n$ and $f_{n+1} : \FU(X_n) \to \ell_{n+1}$ be the witness function associated to $(\ell_n, f_n, F_n, X_n)$. Note that $f_{n+1}$ is $f$-computable, and a code for~$f_{n+1}$ can be found $f''$-computably in $(\ell_n, f_n, F_n, X_n)$.
\end{proof}

\subsection{Notion of forcing}\label[section]{sec:notion-of-forcing}

We are now ready to define our notion of forcing. It is a tree-like forcing, in the sense that every sufficiently generic filter induces an infinite, finitely branching tree of block sequences, and any infinite path through this tree is the desired solution.

\begin{definition}
A \emph{$\PP$-branch} is a pair $(F, X)$ such that
\begin{itemize}
    \item[(1)] $F$ is a non-empty finite block sequence; $X$ is an infinite block sequence ;
    \item[(2)] $F < X$ and $\FU(F)$ is $f$-monochromatic.
\end{itemize}
\end{definition}

One can think of a $\PP$-branch as a Mathias condition with some $f$-monochroma\-ticity constraints. 
Given a $\PP$-branch $(F, X)$, its \emph{denotation} $[F, X]$ is the class of (finite or infinite) block sequences
$$
[F, X] = \{ Z \subseteq \FU(F \cup X) : F \preceq Z \wedge \FU(Z) \text{ is $f$-monochromatic } \}
$$
Note that there are some $\PP$-branches such that $[F, X] = \{F\}$. Also note that by our non-emptiness assumption on~$F$, all the elements of $[F,X]$ have the same color of $f$-monochromaticity, namely, the color of $\FU(F)$.

\begin{definition}
A $\PP$-branch $(E, Y)$ \emph{extends} $(F, X)$ (written $(E, Y) \leq (F, X)$) if $F \preceq E$, $Y \subseteq \FU(X)$, and $E \setminus F \subseteq \FU(X)$.
\end{definition}

We now define the notion of $\PP$-condition, which contains sufficiently many $\PP$-branches so that at least one of them has an $\FUT$-solution in its denotation.

\begin{definition}
A \emph{$\PP$-precondition} is a pair $(\I, X)$ such that
\begin{itemize}
    \item[(1)] $\I$ is a finite set of non-empty finite block sequences;
    \item[(2)] for every~$F \in \I$, $(F, X)$ is a $\PP$-branch
    \item[(3)] $X$ is of arithmetic degree;
\end{itemize}
A $\PP$-precondition $(\I, X)$ is \emph{$f$-matching} if 
\begin{quote}
for every $b \in \FU(X)$, there is some~$F \in \I$ such that for each~$a \in \FU(F)$, $f(a) = f(a \cup b) = f(b)$.
\end{quote}
Combined with (2), this implies that $\FU(F \cup \{b\})$ is $f$-monochromatic.\\
A \emph{$\PP$-condition} is an $f$-matching $\PP$-precondition.
\end{definition}

The \emph{denotation} of a $\PP$-precondition $(\I, X)$, is the class $[\I, X] = \bigcup_{F \in \I} [F, X]$.


\begin{remark}
By \Cref{prop:full-matches-arith}, there is a finite block sequence~$F$ and an arithmetic block sequence $X$
such that $F$ full-matches $X$ for~$f$.
Then, letting $\I = \{ \{a\} : a \in F \}$, $(\I, X)$ is a $\PP$-condition.
\end{remark}

\begin{definition}
A $\PP$-precondition $(\J, Y)$ \emph{extends} $(\I, X)$ (written $(\J, Y) \leq (\I, X)$) if for every~$E \in \J$, there is some~$F \in \I$ such that $(E, Y) \leq (F, X)$. 
\end{definition}

We shall see that the $f$-matching property is sufficient to ensure the existence of an $\FUT$-solution in $[\I, X]$ for any $\PP$-condition $(\I, X)$.
The \emph{canonical coloring} associated to a $\PP$-condition $(\I, X)$ is the coloring $\hat f : \Pfin(\NN) \to \I$ defined by
\begin{align}\label{eq:canonical-for-precondition}
\hat f :\ & b \mapsto F \text{ such that } \notag \\
 & \text{ $\FU(F \cup \{X[b]\})$ is $f$-monochromatic}
\end{align}
If there is more than one~$F$ satisfying (\ref{eq:canonical-for-precondition}), pick the least such one for some fixed order.
As we did for the notion of full-match, the following lemma shows the existence of an $\FUT$-solution in $[\I, X]$, using the fact that $\FUT$ classically holds. Because of the assumption, it is only used as an intuition:

\begin{lemma}\label[lemma]{lem:fut-solution}
For every $\PP$-condition $(\I, X)$, there is an infinite block sequence $Z \in [\I, X]$.
\end{lemma}
\begin{proof}
Let $\hat f : \Pfin(\NN) \to \I$ be the canonical coloring associated to~$(\I, X)$. By $\FUT$, there is an infinite block sequence $Y \subseteq \Pfin(\NN)$ such that 
\begin{align}\label{eq:fut-solution1}
    \text{$\FU(Y)$ is $\hat f$-monochromatic for some color~$F \in \I$}
\end{align}
Since $(F, X)$ is a $\PP$-branch $\FU(F)$ is $f$-monochromatic for some color~$\ell^*$. Unfolding the definition of~$\hat f$, combined with (\ref{eq:fut-solution1}), for every~$b \in \FU(Y)$,
\begin{align}\label{eq:fut-solution2}
    \FU(F \cup \{X[b]\}) \text{ is $f$-monochromatic for color } \ell^*
\end{align}
Let $Z = F \cup X[Y]$.
We claim that $\FU(Z)$ is $f$-monochromatic for color~$\ell^*$.
Indeed, fix any $c \in \FU(Z)$. Let $b \in \FU(Y)$ be such that $c \in \FU(F \cup  \{ X[b] \})$. By (\ref{eq:fut-solution2}), $f(c) = \ell^*$. Since $F \preceq Z$ and $\FU(Z)$ is $f$-monochromatic, $Z \in [F, X] \subseteq [\I, X]$.
This completes the proof of \Cref{lem:fut-solution}.
\end{proof}

\subsection{First-jump control}

This section is at the heart of the notion of forcing, by designing a forcing question with the appropriate definability properties. We first define a strong forcing relation for $\Sigma^0_1$ and $\Pi^0_1$ formulas, which holds for \emph{every} filter containing the condition instead of the sufficiently generic filters.

\begin{definition}[Forcing relation]\label[definition]{def:fut-forcing}
Let $(F,X)$ be a $\PP$-branch and $\varphi(G):= \exists n\psi(G\uh_n)$ be a $\Sigma^0_1$-formula. Let
\begin{itemize}
    \item[(1)] $(F,X)\Vdash \varphi(G)$
    iff $\psi(F\uh_n)$ holds for some $n$;
    \item[(2)] $(F,X)\Vdash \neg\varphi(G)$ iff
    for every $Z \in [F,X]$, every $n\in\NN$, $\neg\psi(Z\uh_n)$ holds.
\end{itemize}
We also say $(F, X)$ \emph{forces} $\varphi(G)$ for $(F, X) \Vdash \varphi(G)$.
\end{definition}

It is not true in general that given a $\PP$-branch $(F, X)$ and a $\Sigma^0_1$-formula $\varphi(G)$, the set of $\PP$-branches forcing either $\varphi(G)$ or $\neg \varphi(G)$ is dense below $(F, X)$. We shall however see that if this density property does not hold, then the $\PP$-branch can be removed from the $\PP$-condition without affecting the $f$-matching property.
\smallskip

Given an arbitrary notion of forcing $(\PP, \leq)$, a \emph{forcing question} is a relation $p \qvdash \varphi(G)$ between a condition $p \in \PP$ and a formula $\varphi(G)$ such that if $p \qvdash \varphi(G)$ holds, then there is an extension $q \leq p$ forcing $\varphi(G)$, and if $p \qvdash \varphi(G)$ does not hold, then there is an extension $q \leq p$ forcing $\neg \varphi(G)$. Contrary to the notion of forcing relation, there exist many potential implementations of forcing questions, with various additional definability and combinatorial properties. The computability-theoretic weakness of the generic sets are closely related to the existence of a forcing question whose definitional complexity is similar to the complexity of the formula $\varphi(G)$. Since we work with a tree-like notion of forcing, the concept of forcing question has to be slightly adapted, but the idea remains the same.

For every $\PP$-condition $(\I, X)$, let $\hat f : \Pfin(\NN) \to \I$ be the canonical coloring associated. We write $\Sc(\I, X)$ for some fixed arithmetic Towsner sequence for~$\hat f$. Accordingly, we write $\Tc(\I, X)$ for its corresponding Towsner tree.

\begin{definition}[Forcing question]\label[definition]{def:fut-forcing-question}
Let $(\I,X)$ be a $\PP$-precondition and $F\in \I$.
We write $(\I,X)\qvdash_F \varphi(G)$
iff there exists a level~$n \in \NN$ such that for every finite block sequence $E$ at level~$n$ in~$\Tc(\I, X)$,
there exists some
block sequence $\hat E\subseteq \FU(E)$
such that $\FU(F \cup X[\hat E])$ is $f$-monochromatic and such that
$\varphi(F \cup X[\hat E])$ holds.
\end{definition}

Note that given a $\Sigma^0_1$-formula $\varphi(G, x)$, the set 
$$\{ n \in \NN : (\I,X)\qvdash_F \varphi(G), n \}$$ 
is 
$\Sigma^0_1(\Tc(\I, X))$, and therefore arithmetic since
$\Sc(\I, X)$ is arithmetic. 
This will be used in an essential way to prove the diagonalization lemma (\Cref{lem:diagonalization-requirement}).

Because of the tree-like nature of the notion of forcing, the answer to the forcing question depends on which $\PP$-branch is considered. In order to obtain an extension which decides the formula on each of the $\PP$-branches, one must be careful not to split the other $\PP$-branches of the $\PP$-condition while taking the extension, so that progress is made. \Cref{fut-lemextPi1} and \Cref{fut-lemextSigma1} show that \Cref{def:fut-forcing-question} satisfies the abstract properties of a forcing question.

\begin{lemma}[$\Pi_1^0$-Extension]\label[lemma]{fut-lemextPi1}
Let $(\I,X)$ be a $\PP$-condition
and $F\in \I$.
If $(\I,X)\nqvdash_F \varphi(G)$,
then there exists a $\PP$-condition
$(\I,\hat X)\leq (\I,X)$ such that
$(F,\h X)\Vdash\neg\varphi(G)$.
\end{lemma}
\begin{proof}
By definition of $(\I,X)\nqvdash_F\varphi(G)$, for every
level~$n \in \NN$, there is a finite block sequence $E$ at level~$n$ in~$\Tc(\I, X)$ such that for every 
block sequence $\hat E\subseteq \FU(E)$
for which $\FU(F \cup X[\hat E])$ is $f$-monochromatic, $\varphi(F \cup X[\hat E])$ does not hold.
Let $Q$ be the $\Pi^{0,X}_1$ class of all~$Y \in [\Tc(\I, X)]$ such that
\begin{align}\label{fut-lemext-eq2}
& \text{for every  
finite block sequence $\hat E\subseteq \FU(Y)$ such that}\notag \\ 
& \text{$\FU(F \cup X[\hat E])$ is $f$-monochromatic,} \notag\\
& \text{we have $\neg\varphi(F \cup X[\hat E])$}.
\end{align}
By compactness, the class~$Q$ is non-empty.
By construction of the Towsner tree, every $Y\in Q$ is an infinite block sequence.
By Kreisel's $\Delta^0_2$ basis theorem,
there is some $Y \in Q$ which is $\Delta^0_2(X)$, hence of arithmetic degree.
Let $\hat X = X[Y]$.
Clearly, $$(\I,\hat X)\leq (\I,X)$$ is a $\PP$-precondition. Furthermore, it is $f$-matching since only the reservoir is restricted, so $(\I,\hat X)$ is a $\PP$-condition.
It remains to show that $(F,\hat X)\Vdash\neg\varphi(G)$,
which follows directly from (\ref{fut-lemext-eq2}).
Thus, we are done.
\end{proof}



\begin{lemma}[$\Sigma_1^0$-Extension]\label[lemma]{fut-lemextSigma1}
Let $(\I,X)$ be a $\PP$-condition and $F \in \I$.
If $(\I,X)\qvdash_F \varphi(G)$,
then there exists a $\PP$-condition $(\h \I,\h X)\leq (\I,X)$
such that for every~$\h F \in \h \I \setminus (\I \setminus \{F\})$, $\h F \succeq F$ and 
$(\h F,\h X)\Vdash\varphi(G)$.
\end{lemma}
\begin{proof}
Unfolding the definition of
$(\I,X)\qvdash_F\varphi(G)$, there is a level $n \in \NN$
such that for every $E$ in level $n$ of  $\Tc(\I, X)$,
\begin{align}\label{lemextSigma1-eq16}
& \text{there exists some  
block sequence $\hat E\subseteq \FU(E)$ 
such that}\notag \\ 
& \text{$\FU(F \cup X[\hat E])$ is $f$-monochromatic,} \notag\\
& \text{and $\varphi(F \cup X[\hat E])$ holds}.
\end{align}
Let $W$ be the set of all $\h E$ satisfying (\ref{lemextSigma1-eq16}).
Say that 
$$
\Sc(\I, X) = \{ (\ell_0, f_0, F_0, X_0), (\ell_1, f_1, F_1, X_1), \cdots \}
$$
where $\ell_0 = \I$ and $f_0 : \Pfin(\NN) \to \I$ is the canonical coloring~$\h f$ associated to $(\I, X)$.
Let
\begin{align}\label{lemextSigma1-eq17}
\h \I & := (\I \setminus \{F\}) \cup \left\{ F \cup X[\h E] : \h E \in W \right\} \\
\h X & := X[X_n]\nonumber
\end{align}
We show that $(\h \I,\h X)$ is the desired extension.
It is obviously a $\PP$-precondition.
By definition of $\h I$ (see also (\ref{lemextSigma1-eq16})),
for every $\h F\in \h \I \setminus (\I \setminus \{F\})$, $\h F \succeq F$ and $\varphi(\h F)$ holds, so $(\h F,\h X)\Vdash\varphi(G)$.
Thus, it remains to show that
\begin{claim}\label[claim]{lemextSigma1-claim4}
$(\h \I,\h X)$ is $f$-matching.
\end{claim}
\begin{proof}
Let 
$b\in \FU(\h X)$
it suffices to find some $\h F \in \h \I$ such that 
\begin{align}\label{lemextSigma1-eq20}
\FU(\h F \cup \{b\})\text{ is $f$-monochromatic.}
\end{align}
Let 
$\hat b\in \FU(X_n)$
be such that
$$b = X[\h b].$$
If $\h f(\h b) \neq F$, then $\h f(\h b) \in (\I \setminus \{F\}) \subseteq \h \I$, and by (\ref{eq:canonical-for-precondition}),
\begin{align}\label{lemextSigma1-eq20}
\FU(\h f(\h b) \cup \{X[\h b]\})\text{ is $f$-monochromatic.}
\end{align}
In this case, we are done by letting $\h F = \h f(\h b)$. So suppose
\begin{align}\label{newlemextSigma1-good-color}
\h f(\h b) = F
\end{align}
By definition of a $\PP$-branch, $F$ is non-empty and $\FU(F)$ is $f$-monochromatic for some color~$\ell^* < \ell$.
By the fundamental property of the Towsner tree $\Tc(\I, X)$ (\Cref{prop:towsner-tree}) applied to the canonical coloring~$\h f$ and the set $\hat b$, there is some finite block sequence~$E$ at level~$n$ in $\Tc(\I, X)$ such that
\begin{align}\label{newlemextSigma1-eq2}
\text{for every $a \in \FU(E)$, $\h f(a \cup \h b) = \h f(\h b) = F$.} 
\end{align}
By (\ref{eq:canonical-for-precondition}), for every $a \in \FU(E)$
\begin{align}\label{newlemextSigma1-eq3}
\text{$\FU(F \cup \{X[a \cup \h b]\})$ and $\FU(F \cup \{X[\h b]\})$} \notag\\
\text{are $f$-monochromatic for color~$\ell^*$}
\end{align}
Since $b = X[\h b]$ and $X[a \cup \h b] = X[a] \cup X[\h b]$, for every $a \in \FU(E)$
\begin{align}\label{newlemextSigma1-eq5}
\text{$\FU(F \cup \{X[a] \cup b\})$ and $\FU(F \cup \{b\})$} \notag\\
\text{are $f$-monochromatic for color~$\ell^*$}
\end{align}
By (\ref{lemextSigma1-eq16}), there is some $\h E \in W$ such that $\h E \subseteq \FU(E)$. In particular, 
\begin{align}\label{newlemextSigma1-eq4}
\text{$\FU(F \cup X[\h E])$ is $f$-monochromatic for color~$\ell^*$}
\end{align}
For any element $c \in \FU(F \cup X[\h E] \cup \{b\})$, one of the following holds:
\begin{itemize}
    \item[(i)] $c \in \FU(F \cup X[\h E])$. By (\ref{newlemextSigma1-eq4}), $f(c) = \ell^*$;
    \item[(ii)] $c \in \FU(F \cup \{X[a] \cup b\})$ with $a \in \FU(\h E)$. By (\ref{newlemextSigma1-eq5}), $f(c) = \ell^*$;
    \item[(iii)] $c \in \FU(F \cup \{b\})$. Again, by (\ref{newlemextSigma1-eq5}), $f(c) = \ell^*$;
\end{itemize}
In all cases we are done by letting $\hat F = F \cup X[\h E] \in \h \I$.
This completes the proof of \Cref{lemextSigma1-claim4}.
\end{proof}
This completes the proof of \Cref{fut-lemextSigma1}.
\end{proof}

\subsection{Positive requirements}

As mentioned, some $\PP$-branches do not contain infinite block sequences in their denotation.
We are going to see that any such $\PP$-branch can be removed from a $\PP$-condition, and therefore that for every sufficiently generic decreasing sequence of $\PP$-conditions, the intersection of the denotations contains only infinite block sequences.

\begin{lemma}\label[lemma]{lem:positive-requirement}
Let $(\I, X)$ be a $\PP$-condition and $F \in \I$. There is a $\PP$-condition $(\h \I, \h X) \leq (\I, X)$ such that for every~$\h F \in \h \I \setminus (\I \setminus \{F\})$, $\h F \succeq F$ and $\card \h F > \card F$.
\end{lemma}
\begin{proof}
Consider the following $\Sigma^0_1$-formula $\varphi(G)$:
\begin{align*}
\text{\qt{$G$ contains at least~$\card F + 1$ elements.}}
\end{align*}
We have two cases:
\smallskip

\noindent
Case 1: $(\I, X) \qvdash_F \varphi(G)$. By \Cref{fut-lemextSigma1}, there is a $\PP$-condition $(\h \I, \h X) \leq (\I, X)$ such that for every~$\h F \in \h \I \setminus (\I \setminus \{F\})$, $\h F \succeq F$ and $(\h F, \h X) \Vdash \varphi(G)$. By \Cref{def:fut-forcing}(1), $\varphi(\h F)$ holds, so unfolding the definition of~$\varphi$, $\card \h F > \card F$ and we are done.
\smallskip

\noindent
Case 2: $(\I, X) \nqvdash_F \varphi(G)$. By \Cref{fut-lemextPi1}, there is a $\PP$-condition $(\I, \h X) \leq (\I, X)$ such that $(F, \h X) \Vdash \neg \varphi(G)$. We claim that $(\I \setminus \{F\}, \h X)$ is a $\PP$-condition.
It is clearly a $\PP$-precondition, so we need to check that $(\I \setminus \{F\}, \h X)$ is $f$-matching.
Pick some~$b \in \FU(\h X)$. Since $(\I, \h X)$ is $f$-matching, there is some~$E \in \I$ such that
\begin{align}
\text{$\FU(E \cup \{b\})$ is $f$-monochromatic}
\end{align}
If $E = F$, then $F \cup \{b\} \in [F, \h X]$ and $\varphi(F \cup \{b\})$ holds. So by \Cref{def:fut-forcing}(2), $(F, \h X) \not\Vdash \neg \varphi(G)$, contradicting our assumption. So $E \neq F$, hence $E \in \I \setminus \{F\}$.
This completes our proof of \Cref{lem:positive-requirement}.
\end{proof}

Iterating \Cref{lem:positive-requirement} on each stem successively, we obtain:

\begin{lemma}\label[lemma]{lem:positive-requirement-full}
Let $(\I, X)$ be a $\PP$-condition. There is a $\PP$-condition $(\h \I, \h X) \leq (\I, X)$ such that for every~$\h F \in \h \I$ and $F \in \I$ with $\h F \succeq F$, $\card \h F > \card F$.
\end{lemma}

\subsection{Diagonalization requirements}

Fix a non-arithmetic set~$C$. We will ensure that $C$ is not $G$-c.e.\ by satisfying the following requirements for every Turing index~$e \in \NN$:
\begin{align*}
\Rc^C_e & : W^G_e \neq C
\end{align*}
We shall actually satisfy each requirement in the following strong sense:

\begin{definition}
Given a $\PP$-condition $(\I, X)$ and $F \in \I$, we say that \emph{$(F, X)$ satisfies $\Rc^C_e$} if one of the following holds:
\begin{itemize}
    \item[(1)] Either $(F, X) \Vdash x \in W^G_e$ for some~$x \not \in C$;
    \item[(2)] Or $(F, X) \Vdash x \not \in W^G_e$ for some~$x \in C$.
\end{itemize}
\end{definition}

Note that this uniformity will not be used in the cone avoidance argument.

\begin{lemma}\label[lemma]{lem:diagonalization-requirement}
Let $(\I, X)$ be a $\PP$-condition, $C$ be a non-arithmetic set and $F \in \I$. There is a $\PP$-condition $(\h \I, \h X) \leq (\I, X)$ such that for every~$\h F \in \h \I \setminus (\I \setminus \{F\})$, $\h F \succeq F$ and $(\h F, \h X)$ satisfies $\Rc^C_e$.
\end{lemma}
\begin{proof}
Consider the following $\Sigma^0_1(\Tc(\I, X))$ set:
\begin{align*}
W = \{ x \in \NN : (\I, X) \qvdash_F x \in W_e^G \}
\end{align*}
Since $W$ is arithmetic, but $C$ is not, $W \neq C$.
We therefore have two cases:
\smallskip

\noindent
Case 1: $x \in W$ for some~$x \not \in C$. Unfolding the definition of~$W$, 
$$(\I, X) \qvdash_F x \in W_e^G.$$
By \Cref{fut-lemextSigma1}, there is a $\PP$-condition $(\h \I, \h X) \leq (\I, X)$ such that for every~$\h F \in \h \I \setminus (\I \setminus \{F\})$, $\h F \succeq F$ and $(\h F, \h X) \Vdash x \in W_e^G$, and we are done.
\smallskip

\noindent
Case 2: $x \not \in W$ for some~$x \in C$. Unfolding the definition of~$W$, 
$$(\I, X) \nqvdash_F x \in W_e^G.$$
By \Cref{fut-lemextPi1}, there is a $\PP$-condition $(\I, \h X) \leq (\I, X)$ such that $(F, \h X) \Vdash x \not \in W^G_e$.
In both cases, we are done. This completes our proof of \Cref{lem:positive-requirement}.
\end{proof}

Iterating \Cref{lem:diagonalization-requirement} on each stem successively, we obtain:

\begin{lemma}\label[lemma]{lem:diagonalization-requirement-full}
Let $(\I, X)$ be a $\PP$-condition and $C$ be a non-arithmetic set. There is a $\PP$-condition $(\h \I, \h X) \leq (\I, X)$ such that for every~$\h F \in \h \I$, $(\h F, \h X)$ satisfies $\Rc^C_e$.
\end{lemma}

Note that the way $\Rc^G_e$ is satisfied depends on the $\PP$-branch.
We are now ready to prove our main theorem:

\subsection{Proof of \Cref{main:cone-avoidance}}

Let $f : \Pfin(\NN) \to \ell$ be an arithmetic coloring for some~$\ell \geq 1$, and let $C$ be a non-arithmetic set.
By \Cref{lem:positive-requirement-full} and \Cref{lem:diagonalization-requirement-full}, there is an infinite decreasing sequence of $\PP$-conditions
$$
(\I_0, X_0) \geq (\I_1, X_1) \geq \cdots
$$
such that for every~$n \in \NN$,
\begin{itemize}
    \item[(1)] every stem $F \in \I_n$ contains at least~$n$ elements;
    \item[(2)] for every stem $F \in \I_n$, $(F, X_n)$ satisfies $\Rc^C_e$.
\end{itemize}
Let $G$ be an element in $\bigcap_n [\I_n, X_n]$. Note that such a class is non-empty, as it is the intersection of a decreasing sequence of non-empty compact classes. By definition of the denotations, $\FU(G)$ is an $f$-monochromatic block sequence. By (1), $G$ is infinite, and by (2), $C$ is not $G$-c.e.\ and a fortiori $C \not \leq_T G$. This completes the proof of \Cref{main:cone-avoidance}.
\bigskip

The existence of an $f$-computable full-match in general is still an open question, of central importance to the computable analysis of Hindman's theorem. Towsner constructed an $f$-arithmetic full-match by iteratively taking infinite paths through Towsner trees. We now prove that for some computable colorings, the existence of such trees implies the existence of the halting set. This however does not rule out the possibility of a new proof of the existence of a full-match which does not involve Towsner trees. In particular, for the specific coloring constructed by Blass, Hirst and Simpson~\cite{blassLogicalAnalysisTheorems1987} to prove the lower bound of Hindman's theorem, and used in \Cref{prop:towsner-zp}, there always exists a computable full-match, with a witnessing set~$F$ of size 2.

\begin{proposition}\label[proposition]{prop:towsner-zp}
There is a computable coloring $f : \Pfin(\NN) \to 2$ such that
every Towsner sequence computes $\emptyset'$.
\end{proposition}
\begin{proof}
This is actually witnessed by the coloring of Blass, Hirst and Simpson~\cite{blassLogicalAnalysisTheorems1987}, based on the notions of short and very short gaps. 
Given a finite set 
$$a  = \{ s_0 < \dots < s_{n-1} \}$$
a \emph{gap} is a pair $(s_i, s_{i+1})$.
A gap $(s, t)$ is \emph{short} if $\emptyset'_t \uh_s \neq \emptyset' \uh_s$, where $\emptyset'_t$ is the $t$th c.e. approximation of~$\emptyset'$. A gap $(s, t)$ is \emph{very short} if $\emptyset'_t \uh s \neq \emptyset'_{\max a} \uh_s$. We write $\SG(a)$ for the number of short gaps and $\VSG(a)$ for the number of very short gaps in~$a$. 
Note that any very short gap is short, and that being very short is decidable, while being short is $\emptyset'$-decidable. Let $f : \Pfin(\NN) \to 2$ be defined by 
$$f(a) := \text{ parity of } \VSG(a).$$
Let
$$\Sc = \{ (\ell_0, f_0, F_0, X_0), (\ell_1, f_1, F_1, X_1), \cdots \}$$
be a Towsner sequence, with Towsner tree $\Tc_{\Sc}$.

\begin{claim}\label[claim]{claim:towsner-zp1}
For every~$n \in \NN$, there is some finite block sequence $E$ at level~$n$ in $\Tc_{\Sc}$ such that
for every~$a \in \FU(E)$, $\SG(a)$ is even.
\end{claim}
\begin{proof}
Let $b \in X_n$ be sufficiently large so that 
\begin{align}\label{eq:towsner-zp-eq1}
    \text{for every $a \in F_n$, }  \emptyset'_{\min b} \uh_{\max a} = \emptyset' \uh_{\max a}.
\end{align}
By \Cref{prop:towsner-tree}, there is some finite block sequence~$E$ at level~$n$ in $\Tc_{\Sc}$ such that
\begin{align}\label{eq:towsner-zp-eq2}
    \text{for every } a \in \FU(E), f(a \cup b) = f(b).
\end{align}
Fix some~$a \in \FU(E)$. By (\ref{eq:towsner-zp-eq1}),
$$
\VSG(a \cup b) = \SG(a) + \VSG(b)
$$
So by (\ref{eq:towsner-zp-eq2}), $\SG(a)$ is even. This proves \Cref{claim:towsner-zp1}.
\end{proof}

\noindent
Given $a < b \in \Pfin(\NN)$, 
$$
\SG(a \cup b) = \SG(a) + \SG(b) + \begin{cases}
    1 & \text{ if $(\max a, \min b)$ is short}\\
    0 & \text{ otherwise}
\end{cases}
$$
Thus, for any block sequence~$E$ satisfying \Cref{claim:towsner-zp1} and any $a < b \in E$,
\begin{align}\label{eq:towsner-zp-eq3}
\emptyset'_{\min b} \uh_{\max a} = \emptyset' \uh_{\max a}
\end{align}
To decide whether $s \in \emptyset'$, $\Sc$-computably search for some~$n \in \NN$ such that for every~$a \in F_n$, $\max a > s$. Such an~$n$ exists since $F_0 < F_1 < \dots$.
Then, let $t$ be the maximum element of~$\{ \min b : b \in F_{n+1} \}$, and check whether $s \in \emptyset'_t$.
To see this algorithm is correct, let $E$ satisfy \Cref{claim:towsner-zp1}, and let $a$ and $b$ be the unique elements of $E \cap F_n$ and $E \cap F_{n+1}$, respectively. By choice of~$n$, we have $\max a > s$ and by choice of~$t$, $\min b \leq t$, so
$$
\emptyset' \uh_{\max a} \overset{(\ref{eq:towsner-zp-eq3})}=  \emptyset'_{\min b} \uh_{\max a}  \subseteq  \emptyset'_t \uh_{\max a} \subseteq  \emptyset' \uh_{\max a}
$$
So $s \in \emptyset'$ iff $s \in \emptyset'_t$. This completes the proof of \Cref{prop:towsner-zp}.
\end{proof}

\section{Finite union theorem for simple colorings}

In this section, we study restrictions of the finite union theorem for various notions of simplicity.
This is motivated by the fact that the proof of $\RCA_0 \vdash \FUT \to \ACA_0$ involves a degenerate coloring that we call \emph{simple}. We shall see that the restriction of $\FUT$ to simple colorings is equivalent to $\ACA_0$.

\begin{definition}
Let $X = \{ a_0, a_1, \dots \}$ be a block sequence and $f : \FU(X) \to \ell$ be a coloring for some~$\ell \geq 1$.
\begin{itemize}
    \item Given $b_0, b_1 \in \FU(X)$, we write $b_0 \sim b_1$ if $\min b_0 = \min b_1$ and $\max b_0 = \max b_1$. If furthermore $f(b_0) = f(b_1)$, then $b_0 \sim_f b_1$.
    \item The coloring $f$ is \emph{simple} if for every~$n \in \NN$ and every $b_0, b_1 \in \FU(X)$ such that $\max a_n < \min b_0, b_1$ and $b_0 \sim_f b_1$, then $f(a_n \cup b_0) = f(a_n \cup b_1)$.
    \item The coloring $f$ is \emph{color-simple} if for every~$n \in \NN$ and every $b_0, b_1 \in \FU(X)$ such that $\max a_n < \min b_0, b_1$ and $f(b_0) = f(b_1)$, then $f(a_n \cup b_0) = f(a_n \cup b_1)$.
\end{itemize}
\end{definition}

We first see that the coloring of Blass, Hirst and Simpson~\cite{blassLogicalAnalysisTheorems1987} is simple.
Recall that, given a finite set 
$$a  = \{ s_0 < \dots < s_{n-1} \}$$
a \emph{gap} is a pair $(s_i, s_{i+1})$.
A gap $(s, t)$ is \emph{short} if $\emptyset'_t \uh_s \neq \emptyset' \uh_s$, where $\emptyset'_t$ is the $t$th c.e. approximation of~$\emptyset'$. A gap $(s, t)$ is \emph{very short} if $\emptyset'_t \uh s \neq \emptyset'_{\max a} \uh_s$. We write $\SG(a)$ for the number of short gaps and $\VSG(a)$ for the number of very short gaps in~$a$.

\begin{lemma}
The following coloring is simple:
$$f : \Pfin(\NN) \ni a \mapsto \text{ parity of } \VSG(a).$$
\end{lemma}
\begin{proof}
Note that $\Pfin(\NN) = \FU(\{ a_0, a_1, \dots \}$ where $a_n = \{n\}$.
Fix some~$n \in \NN$ and some $b_0, b_1 \in \Pfin(\NN)$ such that $n < \min b_0, b_1$ and $b_0 \sim_f b_1$.
For every $i < 2$,
$$
\VSG(a_n \cup b_i) = \VSG(b_i) + \begin{cases}
    1 & \text{ if } (n, \min b_i) \mbox{ is very small in } a_n \cup b_i\\
    0 & \text{ otherwise}
\end{cases}
$$
Since $b_0 \sim b_1$, $\min b_0 = \min b_1$, $\max b_0 = \max b_1$, so $(n, \min b_0)$ is very small in $a_n \cup b_0$ iff $(n, \min b_1)$ is very small in $a_n \cup b_1$. Moreover, since $b_0 \sim_f b_1$, $f(b_0) = f(b_1)$, so the parity of $\VSG(b_0)$ and of $\VSG(b_1)$ is the same. It follows that $f(a_n \cup b_0) = f(a_n \cup b_1)$.
\end{proof}

\subsection{Color-simple colorings over~$\RCA_0 + \BSig_2$}

Given $X$, $f$ and $a \in \FU(X)$, we write $(\ddagger_a)$ for the property 
\begin{quote}
\qt{for every $b_0, b_1 \in \FU(X)$ such that $\max a < \min b_0, b_1$ and $f(b_0) = f(b_1)$, then $f(a \cup b_0) = f(a \cup b_1)$.}
\end{quote}

It is not clear at first sight that color-simplicity is closed under taking block sub-sequences.
The following lemma shows that it is the case.

\begin{lemma}\label[lemma]{lem:fut-color-simple-to-strong}
A coloring $f : \FU(X) \to \ell$ is color-simple iff for every~$a \in \FU(X)$, the property $(\ddagger_a)$ holds.
\end{lemma}
\begin{proof}
The backward direction is immediate since for every~$n \in \NN$, $a_n \in \FU(X)$.
Suppose now $f$ is color-simple. Given $a \in \FU(X)$, the \emph{union-size} of $a$, written $\us(a)$ is the cardinality of the set~$H$ such that $a = \bigcup_{n \in H} a_n$. We prove by induction on the union-size of~$a$ that $(\ddagger_{a})$ holds. The case $\us(a) = 1$ holds by color-simplicity of~$f$. Let $a \in \FU(X)$ and let $b_0, b_1 \in \FU(X)$ be such that $\max a < \min b_0, b_1$ and $f(b_0) = f(b_1)$. Say $a = a_i \cup a'$ with $a_i < a'$.
Then $\us(a') < \us(a)$, so by induction hypothesis, $(\ddagger_{a'})$ holds, so $f(a' \cup b_0) = f(a' \cup b_1)$.
Let $b_0' = a' \cup b_0$ and $b_1' = a' \cup b_1$. Then $\max a_i < \min(b_0', b_1')$, $\min b_0' = \min b_1' = \min a'$, and $f(b_0') = f(a' \cup b_0) = f(a' \cup b_1) = f(b_1')$. By color-simplicity of~$f$, $f(a \cup b_0) = f(a_i \cup b_0') = f(a_i \cup b_1') = f(a \cup b_1)$, so $(\ddagger_a)$ holds. 
\end{proof}

Given a color-simple coloring $f : \FU(X) \to \ell$, by \Cref{lem:fut-color-simple-to-strong}, any set $a \in \FU(X)$ induces a partial map $\rho_a : \ell \to \ell$ such that for every~$b \in \FU(X)$ with $a < b$, then $f(b) \in \dom \rho_a$ and $f(a \cup b) = \rho_a(f(b))$. Moreover, for any such $a$ and $b$, $\ran \rho_b \subseteq \dom \rho_a$, and $\rho_{a \cup b} = \rho_a \circ \rho_b$.

\begin{proposition}[$\RCA_0 + \BSig_2$]\label[proposition]{lem:fut-color-simple-to-hom}
Let $X = \{ a_0, a_1, \dots \}$ be a computable block sequence and $f : \FU(X) \to \ell$ be a color-simple coloring for some~$\ell \geq 1$. There is a block sequence $Y \subseteq \FU(X)$ such that $f$ avoids at least one color over $\FU(Y)$.
\end{proposition}
\begin{proof}
Suppose there is no such block sequence~$Y$, as otherwise we are done.

\begin{claim}
For every~$a \in \FU(X)$, $\dom \rho_a = \ell$.
\end{claim}
\begin{proof}
Indeed, otherwise, letting $x \not \in \dom \rho_a$ and $i$ be the least such that $a < a_i$, the set $Y = \{ a_j : j > i \}$ would be a block sequence such that $f$ avoids $x$ on $\FU(Y)$, contradicting our assumption. This proves Claim~1.
\end{proof}

It follows from Claim~1 that the map $a \mapsto \rho_a$ is $f \oplus X$-computable. 
By $\RT^1$ (which is equivalent to~$\BSig_2$ over~$\RCA_0$) applied on the $\ell^\ell$-coloring $n \mapsto \rho_{a_n}$, there is some map~$\rho : \ell \to \ell$ and an infinite set~$A = \{ n_0 < n_1 < \dots \} \subseteq \NN$ such that for every~$n \in A$, $\rho_{a_n} = \rho$.

\begin{claim}
$\rho$ is a permutation.
\end{claim}
\begin{proof}
Suppose not. Let $x < \ell$ be such that $x \not \in \operatorname{ran}(\rho)$.
For every~$i \in \NN$, let $\hat a_i = a_{n_{2i}} \cup a_{n_{2i+1}}$.
Let $Y = \langle \hat a_i : i \in \NN \rangle$. Then $f$ avoids color~$x$ on~$\FU(Y)$. 
Indeed, given $a \in \FU(G)$, there is some~$i \in \NN$ such that $a = \hat a_i \cup d$ for some~$d > \hat a_i$.
In particular, $a = a_{n_{2i}} \cup a_{n_{2i+1}} \cup d$. By choice of~$\rho$, $f(a) = f(a_{n_{2i}} \cup a_{n_{2i+1}} \cup d) = \rho(f(a_{n_{2i+1}} \cup d))$. Since $x \not \in \operatorname{ran}(\rho)$, $f(a) \neq x$. This contradicts our assumption, and proves Claim~2.
\end{proof}

Let $k \in \NN$ be such that $\rho^{(k)}$ is the identity permutation. Let $b_i = a_{n_{ki}} \cup \dots \cup a_{n_{k(i+1)-1}}$ and let $Y = \{ b_i : i \in \NN \}$. By an application of~$\RT^1$ to the $\ell$-coloring $i \mapsto f(b_i)$, we can furthermore assume that there is some color~$x < \ell$ such that $f(b_i) = x$ for every~$i \in \NN$. Note that by construction, $\rho_{b_i} = \rho^{(k)}$, hence is the identity function. It follows that for every~$i < j \in \NN$, $f(b_i \cup b_j) = f(b_j) = x$. Thus, $f$ is monochromatic for color~$x$ on~$\FU(Y)$. This completes our proof of \Cref{lem:fut-color-simple-to-hom}.
\end{proof}

\subsection{Simple colorings over~$\ACA_0$}

Given $X$, $f$ and $a \in \FU(X)$, we write $(\dagger_a)$ for the property 
\begin{quote}
\qt{for every $b_0, b_1 \in \FU(X)$ such that $\max a < \min b_0$ and $b_0 \sim_f b_1$, then $f(a \cup b_0) = f(a \cup b_1)$.}
\end{quote}

Here again, it is not clear that simplicity is closed under taking block sub-sequences.
The following lemma shows that it is the case.

\begin{lemma}\label[lemma]{lem:fut-simple-to-strong}
A coloring $f : \FU(X) \to \ell$ is simple iff for every~$a \in \FU(X)$, the property $(\dagger_a)$ holds.
\end{lemma}
\begin{proof}
The backward direction is immediate since for every~$n \in \NN$, $a_n \in \FU(X)$.
Suppose now $f$ is simple. Recall that given $a \in \FU(X)$, the \emph{union-size} of $a$, written $\us(a)$ is the cardinality of the set~$H$ such that $a = \bigcup_{n \in H} a_n$. We prove by induction on the union-size of~$a$ that $(\dagger_{a})$ holds. The case $\us(a) = 1$ holds by simplicity of~$f$. Let $a \in \FU(X)$ and let $b_0, b_1 \in \FU(X)$ be such that $\max a < \min b_0$ and $b_0 \sim_f b_1$. Say $a = a_i \cup a'$ with $a_i < a'$.
Then $\us(a') < \us(a)$, so by induction hypothesis, $(\dagger_{a'})$ holds, so $f(a' \cup b_0) = f(a' \cup b_1)$.
Let $b_0' = a' \cup b_0$ and $b_1' = a' \cup b_1$. Then $\max a_i < \min(b_0',b_1')$, $\min b_0' = \min b_1' = \min a'$, $\max b_0' = \max b_1'=  \max b_0 = \max b_1$, and $f(b_0') = f(a' \cup b_0) = f(a' \cup b_1) = f(b_1')$, so $b_0' \sim_f b_1'$. By simplicity of~$f$, $f(a \cup b_0) = f(a_i \cup b_0') = f(a_i \cup b_1') = f(a \cup b_1)$, so $(\dagger_a)$ holds. 
\end{proof}

\begin{proposition}[$\ACA_0$]\label[proposition]{prop:fut-simple-to-hom}
Let $X = \{ a_0, a_1, \dots \}$ be a computable block sequence and $f : \FU(X) \to \ell$ be a simple coloring for some~$\ell \geq 1$. There is a block sequence $Y \subseteq \FU(X)$ such that $f$ is monochromatic on $\FU(Y)$.
\end{proposition}
\begin{proof}
Fix $X$ and $f$.
Let $A = \{ \min a_n : n \in \NN \}$. For $x < \ell$ Let $g_x : [A]^3 \to \ell$ be defined by $g_x(\min a_i, \min a_j, \min a_k) = f(a_i \cup a_j \cup c \cup a_k)$ if there is some~$c \in \FU(a_j, \dots, a_k)$ such that $f(a_j \cup c \cup a_k) = x$, and $g_x(\min a_i, \min a_j, \min a_k) = 0$ otherwise. Note that $a_j, a_k \in \FU(a_j, \dots, a_k)$, so it might be that $a_j \cup c \cup a_k = a_j \cup a_k$.

By an application of~$\RT^3$ on the product coloring $g_0 \times \cdots g_{\ell-1}$, there is an infinite subset~$B = \{ n_0 < n_1 < \dots \} \subseteq \NN$ such that for each~$x < \ell$, $B$ is $g_x$-homogeneous for some color~$y_x < \ell$. For every~$i \in \NN$, let $\hat a_i = a_{n_{2i}} \cup a_{n_{2i+1}}$, and let $Y = \langle \hat a_i : i \in \NN \rangle$.

\begin{claim}
For every~$i \in B$ and $d \in \FU(Y)$ with $\max a_i < \min d$, then $f(a_i \cup d) = g_{f(d)}(i, j, k)$ for some~$j < k \in B$.
\end{claim}
\begin{proof}
Fix~$i$, $d$. Note that $d = a_j \cup c \cup a_k$ for some $j < k \in Y$ and $c \in \FU(a_j, \dots, a_k)$. Let $x = f(d)$. Then $g_x(i, j, k) = f(a_i \cup a_j \cup c' \cup a_k)$ for some $c' \in \FU(a_j, \dots, a_k)$ such that $f(a_j \cup c' \cup c_k) = x$. Note that $d = a_j \cup c \cup a_k \sim_f a_j \cup c' \cup a_k$,
so by simplicity of~$f$, $f(a_i \cup d) = f(a_i \cup a_j \cup c' \cup a_k) = g_x(i, j, k)$. This proves our Claim~1.
\end{proof}

\begin{claim}
For every~$i \in B$ and every $b_0, b_1 \in \FU(Y)$ with $\max a_i < \min b_0, b_1$ and $f(b_0) = f(b_1)$, then $f(a_i \cup b_0) = f(a_i \cup b_1)$.
\end{claim}
\begin{proof}
Fix~$i$, $b_0$ and $b_1$. Let $x = f(b_0) = f(b_1)$. By Claim 1, there are some~$j < k \in B$ and some~$u < v \in B$ such that $f(a_i \cup b_0) = g_x(i, j, k)$ and $f(a_i \cup b_1) = g_x(i, u, v)$. Since $B$ is $g_x$-homogeneous, then $g_x(i, u, v) = g_x(i, j, k)$, so $f(a_i \cup b_0) = f(a_i \cup b_1)$. This proves Claim 2.
\end{proof}

\begin{claim}
$f$ is color-simple on~$\FU(Y)$.
\end{claim}
\begin{proof}
Fix some $i \in \NN$ and some~$b_0, b_1 \in \FU(Y)$ with $\max \hat a_i < \min b_0, b_1$ and $f(b_0) = f(b_1)$.
Let $n, m \in B$ be such that $\hat a_i = a_n \cup a_m$. By Claim~2, $f(a_m \cup b_0) = f(a_m \cup b_1)$, so still by Claim~2, $f(a_n \cup a_m \cup b_0) = f(a_n \cup a_m \cup b_1)$. This proves Claim~3.
\end{proof}

By \Cref{lem:fut-color-simple-to-hom}, there is a block sequence $Z \subseteq \FU(Y)$ such that $f$ is monochromatic on~$\FU(Z)$. This completes our proof of \Cref{prop:fut-simple-to-hom}.
\end{proof}

\bibliographystyle{plain}
\bibliography{biblio}

\end{document}